\documentclass[12pt]{article}
\usepackage{setspace} 
\usepackage[dvips]{graphicx} 
\usepackage{amsmath,amsthm,amsfonts}
\RequirePackage[colorlinks,citecolor=blue,urlcolor=blue,linkcolor=blue]{hyperref}
\hypersetup{
colorlinks = true,
citecolor=blue,
urlcolor=blue,
linkcolor=blue,
pdfauthor = {Friedrich Hubalek, Alexey Kuznetsov},
pdfkeywords = {stable processes, supremum, Mellin transform, double Gamma function, Liouville numbers, continued fractions},
pdftitle = {A convergent series representation for the density of the supremum of a stable process},
pdfpagemode = UseNone
}
    \def\qed{\hfill$\sqcap\kern-8.0pt\hbox{$\sqcup$}$\\}
    \def\beq{\begin{eqnarray}}
    \def\eeq{\end{eqnarray}}
    \def\beqq{\begin{eqnarray*}}
    \def\eeqq{\end{eqnarray*}}
    \def\zz{{\mathbb Z}}
    
    \def\q{{\mathbb Q}}
    
    \def\re{\textnormal {Re}}
    \def\im{\textnormal {Im}}
    \def\p{{\mathbb P}}
    \def\e{{\mathbb E}}
    \def\r{{\mathbb R}}
    \def\aa{\mathcal A}
    \def\mm{{\mathcal M}}
    \def\c{{\mathbb C}}
    \def\cc{\mathcal C}
    	
    \def\d{{\textnormal d}}
    \def\i{{\textnormal i}}

	\newtheorem{theorem}{Theorem}
	\newtheorem{lemma}{Lemma}
	\newtheorem{proposition}{Proposition}
	
	\newtheorem{definition}{Definition}

\title{A convergent series representation for the density of the supremum of a stable process}
\author{F. Hubalek 
\\ \\
Financial and Actuarial Mathematics \\
Vienna University of Technology \\
Wiedner Hauptstra\ss e 8 / 105-1 \\
A-1040 Vienna, Austria \\
\and 
A. Kuznetsov
\thanks{{Research supported by the
Natural Sciences and Engineering Research Council of Canada.}}  \\ \\
Dept. of Mathematics and Statistics\\  York University \\
4700 Keele Street 
\\Toronto, ON \\ M3J 1P3,  Canada 
 }

\date{current version: January 10, 2011}

\begin{document}

\maketitle

\begin{abstract}
\bigskip
We study the density of the supremum of a strictly stable L\'evy process. We prove that 
for almost all values of the index $\alpha$ -- except for a dense set of Lebesgue measure zero -- the asymptotic series which were obtained in
\cite{Kuz2010} are in fact absolutely convergent series representations for the density of the supremum.
\end{abstract}

{\vskip 0.5cm}
 \noindent {\it Keywords}: stable processes, supremum, Mellin transform, double Gamma function, Liouville numbers, continued fractions 
{\vskip 0.5cm}
 \noindent {\it 2000 Mathematics Subject Classification }: 60G52 

\newpage


\section{Introduction}


 Let $X$ be an  $\alpha$-stable process $X$ indexed by parameters $(\alpha,\rho)$, where as usual $\rho=\p(X_1>0)$. For the definition and properties
of stable processes and stable distributions we refer to \cite{Bertoin}, \cite{Kyprianou}, \cite{Zolotarev1986} and \cite{Zolotarev1957}.  
The admissible set of parameters $(\alpha,\rho)$ is defined as
\beqq
\aa=\{\alpha \in (0,1), \; \rho \in (0,1)\} \cup \{\alpha=1, \; \rho={\textstyle{\frac12}} \}  \cup \{\alpha\in(1,2), \; \rho \in [1-\alpha^{-1}, \alpha^{-1}]\}.
\eeqq
Note that we exclude the case when $\alpha \in (0,1)$ and $\rho=1$ \{$\rho=0$\}, as in this case the process $X$ \{$-X$\} is a subordinator and the distribution of extrema is trivial.  When $\alpha \in (1,2)$ and $\rho =1-\alpha^{-1}$ \{$\rho=\alpha^{-1}$\} the process $X$ is spectrally positive \{negative\}. In this case we have complete information about the distribution of extrema due to the work of Bingham \cite{Bingham1975}, Doney \cite{Doney2008}, Bernyk, Dalang and Peskir \cite{Bernyk2008} and Patie \cite{Patie2009}.

We are interested in the distribution of the supremum of the process $X$, defined as
\beqq
S_t=\sup\{X_u\; : 0\le u \le t\}.
\eeqq
Note that due to the scaling property of stable processes, we have $S_t \stackrel{d}{=} t^{\frac{1}{\alpha}} S_1$, 
thus it is sufficient to study the distribution of $S_1$. The main object of interest in this paper will be the  probability density function 
\beqq
 p(x)= \frac{\d}{\d x} \p(S_1 \le x), \;\;\; x>0.
\eeqq

The complete description of the asymptotic behavior of $p(x)$ as $x\to 0^+$ and as $x\to +\infty$ is provided by the following Theorem, which was proved in \cite{Kuz2010}. This Theorem generalizes results for the spectrally one-sided case obtained by Doney \cite{Doney2008} and Patie \cite{Patie2009} and the results on the two-sided case, which were proved by Doney and Savov \cite{Doney2010}.

\begin{theorem}(Theorem 9, \cite{Kuz2010})\label{thm1}  Assume that $\alpha \notin \q$.  
Define sequences $\{a_{m,n}\}_{m\ge 0,n\ge 0}$ and  $\{b_{m,n}\}_{m\ge 0,n\ge 1}$ as
 \beq\label{def_a_mn}
a_{m,n}=\frac{(-1)^{m+n} }{\Gamma\left(1-\rho-n-\frac{m}{\alpha}\right)\Gamma(\alpha\rho+m+\alpha n)}
\prod\limits_{j=1}^{m} \frac{\sin\left(\frac{\pi}{\alpha} \left( \alpha \rho+ j-1 \right)\right)} {\sin\left(\frac{\pi j}{\alpha} \right)} 
\prod\limits_{j=1}^{n} \frac{\sin(\pi \alpha (\rho+j-1))}{\sin(\pi \alpha j)},
\eeq
\beq\label{def_b_mn}
b_{m,n}=\frac{\Gamma\left(1-\rho-n-\frac{m}{\alpha}\right)\Gamma(\alpha\rho+m+\alpha n) }{\Gamma\left(1+n+\frac{m}{\alpha}\right)\Gamma(-m-\alpha n)}
a_{m,n}.
\eeq
Then we have the following asymptotic expansions:
\beq \label{eqn_p_0_asympt}
p(x) &\sim& x^{\alpha\rho-1} \sum\limits_{n\ge 0} \sum\limits_{m\ge 0} a_{m,n} x^{ m+\alpha n}, \;\;\; x\to 0^+,  \\
\label{eqn_p_infty_asympt}
p(x) &\sim& x^{-1-\alpha } \sum\limits_{n\ge 0} \sum\limits_{m\ge 0}b_{m,n+1} x^{-m-\alpha n}, \;\;\; x\to +\infty.
\eeq 
\end{theorem}

There is a very important subclass of stable processes, for which the above result can be considerably strengthened. The following family of stable processes was first introduced by Doney \cite{Doney1987}:
\begin{definition}
For $k,l \in \zz$ define $\cc_{k,l}$ as the class of stable processes with parameters $(\alpha,\rho) \in \aa$ satisfying 
\beq\label{eqn_def_Ckl}
\rho+k=\frac{l}{\alpha}.
\eeq 
\end{definition} 
It turns out (see Theorem 10 in \cite{Kuz2010}), that when a process $X \in \cc_{k,l}$ for some integers $k$ and $l$, then 
the coefficients $a_{m,n}$ and $b_{m,n}$ can be simplified so that they involve 
only finite products of lentgh not greater than $k$ or $l$. Also, in this case the coefficients are well defined even for rational $\alpha$. Moreover, when $\alpha \in (0,1)$ \{$\alpha \in (1,2)$\}, the series in the right-hand side of  (\ref{eqn_p_infty_asympt}) \{ (\ref{eqn_p_0_asympt}) \} converges to $p(x)$ for all $x>0$.

It is not hard to see that parameters $(\alpha,\rho) \in \aa$ which satisfy relation (\ref{eqn_def_Ckl}) form a dense set of Lebesgue measure zero. Therefore, Theorem 10 in \cite{Kuz2010} gives a convergent series representation for $p(x)$ for a dense set of parameters. The main goal of this paper is to extend this result and to prove that in fact we have a convergent series representation for $p(x)$ for {\it almost all} parameters $(\alpha,\rho)$, except when $\alpha \in \q$ or 
when $\alpha$ can be approximated by rational numbers extremely well.


\section{Main Results}


The following set of real transcendental numbers, which can be approximated by rational numbers extremely well, was introduced in \cite{Kuz2010}. 

\begin{definition}\label{def_set_L}
 Let ${\mathcal L}$ be the set of all real irrational numbers $x$, for which there exists a constant $b>1$ such that the inequality
\beq\label{eqn_def_set_L}
 \bigg| x -\frac{p}{q} \bigg| < \frac{1}{b^{q}}
\eeq 
is satisfied for infinitely many integers $p$ and $q$.
\end{definition}

It is clear that ${\mathcal L}$ is a proper subset of the set of Liouville numbers, which are defined by the following, weaker condition: for all $n\ge 1$ the 
inequality
\beqq
 \bigg| x-\frac{p}{q} \bigg| < \frac{1}{q^{n}}
\eeqq
is satisfied for infinitely many integers $p$ and $q$.  A celebrated result by Liouville states
 that any algebraic number is not a Liouville number, but this is also true for many other numbers. 
In fact, almost every number is not a Liouville number, 
as the set of Liouville numbers, while being dense in $\r$, has zero Lebesgue measure  (see Theorem 32 in \cite{Khinchin}). Therefore, the Lebesgue measure of 
 ${\mathcal L}$ is also zero.

As we will see later, the structure of the set ${\mathcal L}$ can be described in terms of continued fraction representation of real numbers. We present
here the essential results from the theory of continued fractions, which will be needed later. 

The continued fraction representation of  a real number $x$ is defined as
\beqq	
 x=[a_0;a_1,a_2,\dots]=a_0+\cfrac{1}{a_1+\cfrac{1}{a_2+ \dots }}
 \eeqq
where $a_0 \in {\mathbb Z}$ and $a_i \in {\mathbb N}$ for $i\ge 1$. For $x\notin \q$ the continued fraction has infinitely many terms; truncating it after $n$ steps gives us   a rational number $p_n/q_n=[a_0;a_1,a_2,...,a_n]$, which is called the $n$-th convergent.
The sequence of coprime integer numbers $p_n$ and $q_n$ can be computed recursively
\beq\label{continued_fraction_recurrence}
\begin{cases}
p_n&=a_n p_{n-1} + p_{n-2}, \;\;\; p_{-1}=1, \;\;\; p_{-2}=0, \\
q_n&=a_n q_{n-1} + q_{n-2}, \;\;\; q_{-1}=0, \;\;\; q_{-2}=1.
\end{cases}
\eeq
Theorem 17 in \cite{Khinchin} tells us that the $n$-th convergent $p_n/q_n$ provides the best rational approximation to $x$ in the following sense
\beqq
\bigg | x - \frac{p_n}{q_n} \bigg| =  \min \left\{ \bigg | x - \frac{p}{q} \bigg|\; :  \; p\in {\mathbb Z}, 1 \le q \le q_n \right\}.
\eeqq
There is also a converse result (see Theorem 19 in \cite{Khinchin}): 
if integers $p$ and $q$ satisfy 
\beq\label{converse_result}
\bigg  |x - \frac{p}{q} \bigg |< \frac{1}{2q^2},
\eeq
then $p=p_n$ and $q=q_n$ for some $n$. The error of the best rational approximation 
is bounded from above and below as follows 
\beq\label{eqn_nth_conv_error}
\frac{1}{q_n (q_n+q_{n+1})} <  \bigg | x - \frac{p_n}{q_n} \bigg| < \frac{1}{q_n q_{n+1}},
\eeq
see Theorems 9 and 13 in \cite{Khinchin}.

The next proposition gives us an insight into the arithmetic structure of the set ${\mathcal L}$. We will use the following notation: $\{x\}\in [0,1)$ denotes the 
fractional part of $x$, i.e. the distance from $x$ to the closest integer not greater than $x$;  and $\langle x \rangle =\min(\{x\},\{-x\})=\min(|x-n|: n \in {\mathbb Z})$, i.e. the distance from $x$ to the closest integer.  
\begin{proposition}\label{prop_set_L}

${}$
 \begin{itemize}
  \item[(i)] If $x\in {\mathcal L}$ then $z x \in {\mathcal L}$ and $z+x \in {\mathcal L}$ for all  $z \in {\mathbb Q}\setminus\{0\}$.
  \item[(ii)]  $x \in {\mathcal L}$ if and only if $x^{-1} \in {\mathcal L}$.
  \item[(iii)] Let $x=[a_0;a_1,a_2,\dots]$. Then $x \in {\mathcal L}$ if and only if there exists a constant $b>1$ such that the inequality
$a_{n+1}>b^{q_n}$ is satisfied for infinitely many $n$. 
 \item[(iv)] $x \notin {\mathcal L} \cup \q$ if and only if
\beq\label{set_L_lim_condition}
 \lim\limits_{q\to +\infty} \frac{\ln \langle qx \rangle }{q} =0.
\eeq 
 \end{itemize}
\end{proposition}
\begin{proof}
Statement (i) follows immediately from the Definition \ref{def_set_L}.
 Statement (ii) also can be derived from the Definition \ref{def_set_L}, however it easily follows from (iii) due to the following simple
property of continued fractions: if $x > 1$ and 
$x=[a_0;a_1,a_2,\dots]$, then $x^{-1} = [0;a_0,a_1,\dots]$. Thus we only need to prove (iii) and (iv).

Let us prove the ``if'' part of (iii). Assume that there exists $b>1$ such that 
\beq\label{ineq_an_cbn}
a_{n+1}> b^{q_n}
\eeq
 for infinitely many $n$. Let us consider such an index $n$. 
Using (\ref{ineq_an_cbn}) and the recurrence relation (\ref{continued_fraction_recurrence}) we find that $q_{n+1}>  q_n a_{n+1} > q_n b^{ q_n}$, which together 
wih (\ref{eqn_nth_conv_error}) implies that
\beqq
 \bigg| x - \frac{p_n}{q_n} \bigg | < \frac{1}{q_n q_{n+1}} < \frac{1}{q_n^2  b^{q_n} } < \frac{1}{b^{q_n}}.
\eeqq
Since the above inequality is true for infinite many $n$, we conclude that $x \in {\mathcal L}$. 

Now, let us prove the ``only if'' part of (iii). Let us assume that $x \in {\mathcal L}$. Then there exists a constant $b>1$ such that the inequality (\ref{eqn_def_set_L}) is satisfied for infinitely 
many pairs $(p,q)$. When $q$ is large enough we have $b^{-q}<1/(2q^2)$, thus $p$ and $q$ satisfy (\ref{converse_result}) and we 
 conclude that the pair $(p,q)$ must coincide with some continued fraction convergent $(p_n,q_n)$. 
Therefore, from (\ref{eqn_nth_conv_error}) we find
\beqq
 \frac{1}{q_n ( q_n+q_{n+1})} < \bigg | x - \frac{p_n}{q_n} \bigg | < \frac{1}{b^{q_n}}.
\eeqq
Using the recurrence relation (\ref{continued_fraction_recurrence}) and the above inequality we conclude that 
\beqq
 \frac{1}{q_n^2 ( a_{n+1} + 2)} <  \frac{1}{q_n ( q_n + a_{n+1} q_n + q_{n-1})}= \frac{1}{q_n ( q_n+q_{n+1})} < \frac{1}{b^{q_n}},
\eeqq
therefore
\beqq
a_{n+1} > \frac{b^{q_n}}{q_n^2}-2.
\eeqq
It is clear that if $q_n$ is large enough then
\beqq
\frac{b^{q_n}}{q_n^2}-2 > {(\sqrt{b})^{q_n}},
\eeqq
Thus we have found infinitely many indices $n$ such that that $a_{n+1} > (\sqrt{b})^{q_n}$. This ends the proof of (iii).

Let us prove the ``if'' part of (iv). Since $x \notin {\mathcal L} \cup \q$ we know that for every $b>1$, all $q \in N$ sufficiently large 
and all $p \in {\mathbb Z}$ we have
\beqq
 |qx-p|>\frac{q}{b^q}>\frac{1}{b^q},  
\eeqq
therefore $\ln \langle qx \rangle > - q \ln(b)$ for all $q$ sufficiently large, which shows that 
\beqq
 \liminf\limits_{q\to +\infty}  \frac{\ln \langle qx \rangle }{q} \ge 0.
\eeqq 
Since we also have $\ln \langle qx \rangle / q < 0$ for all $q$, it implies condition (\ref{set_L_lim_condition}). The ``only if'' part of (iv) can be verified in
exactly the same way and  we leave the details to the reader.
\end{proof}

Proposition \ref{prop_set_L} shows  that the set ${\mathcal L}$ has quite an interesting structure. First of all, property (iii) gives us  a simple method to
construct a number $x \in {\mathcal L}$ (just define recursively $a_{n+1}=2^{q_n}$), therefore this set is not empty.  The set ${\mathcal L}$ is closed under addition and multiplication by rational numbers,
therefore it is dense in $\r$. It is a subset of the set of Liouville numbers, therefore it has Lebesgue measure zero.  In fact, the Hausdorff dimension of the set
of Liouville numbers is also zero (see Theorem 2.4 in \cite{Oxtoby}), and of course the same is true for ${\mathcal L}$. 

We would like to stress that 
the elements of the set ${\mathcal L}$ are rather unusual transcendental numbers with quite extreme arithmetic properties.  
The reason for this lies in propety (iii) of the 
above proposition. Assume that $x$ is an irrational number. Theorem 12 in \cite{Khinchin} tells us that 
for every irrational number $x$
\beqq
q_n \ge 2^{\frac{n-1}2}, \;\;\; n\ge 2,
\eeqq
 therefore using Proposition \ref{prop_set_L}, (iii) we find that for every number $x \in {\mathcal L}$ there exists a constant $b>1$ such that the inequality
\beq\label{increase_set_L}
a_n>b^{(\sqrt{2})^n}
\eeq
is satisfied for infinitely many $n$. Yet one can prove that for all $x$ except a set of Lebesgue measure zero and all $b>1$ 
\beq\label{increase_general}
a_n = O(b^n), \;\;\; n\to +\infty. 
\eeq 
This result follows from the famous L\'evy-Khinchin theorem, which states that for almost all $x$ it is true that $\sqrt[n]{q_n} \to \gamma$ as 
$n\to +\infty$, where $\gamma = \exp ( \pi^2/(12 \ln(2)))$ (see Theorem 31 and the footnote on page 66 in \cite{Khinchin}). 
We see now that there is a very large gap between the growth rate of coefficients $a_n$ for elements of ${\mathcal L}$ given by (\ref{increase_set_L}) and
almost all other real numbers given by (\ref{increase_general}). This shows that the elements of the set ${\mathcal L}$ should be rather exceptional numbers.

Our main result in this paper is the following Theorem, which gives an absolutely convergent series representation for $p(x)$ for every irrational $\alpha$ which is not in the set ${\mathcal L}$. 

\begin{theorem}\label{thm_main}
Assume that $\alpha \notin {\mathcal L} \cup \q$. Then for all $x>0$ 
\begin{equation}
p(x)=
\begin{cases} \label{eqn_p_0_infty}
 & \displaystyle x^{-1-\alpha } \sum\limits_{n\ge 0} \sum\limits_{m\ge 0}b_{m,n+1} x^{-m-\alpha n}, \;\; {\textnormal { if }} \alpha \in (0,1), \\ \\
 & \displaystyle x^{\alpha\rho-1} \sum\limits_{n\ge 0} \sum\limits_{m\ge 0} a_{m,n} x^{ m+\alpha n}, \;\;\;\;\;\;\;  {\textnormal { if }} \alpha \in (1,2),
\end{cases}
\end{equation}	
where $a_{m,n}$ and $b_{m,n}$ are defined by (\ref{def_a_mn}) and (\ref{def_b_mn}).
\end{theorem}

Before we prove Theorem \ref{thm_main}, let us establish the following result.
\begin{lemma}\label{lemma_trig_asymptotics}
 Assume that $x \notin {\mathcal L} \cup \q$.  Then 
\beq\label{asymptotics_sec}
  \prod\limits_{l=1}^k   |\sec(\pi l x  ) | &=& 2^{k+o(k)}, \;\;\; k \to +\infty, \\
\label{asymptotics_csc}
\prod\limits_{l=1}^k   |\csc(\pi l x  ) | &=& 2^{k+o(k)}, \;\;\; k \to +\infty.
\eeq
\end{lemma}
\begin{proof} 
Let us prove (\ref{asymptotics_sec}). We use the following easily verified fact
\beqq
 \bigg |\{y\}-\frac12 \bigg |=\left \langle y - \frac 12 \right \rangle \ge  \frac{\langle 2 y \rangle}{2}, \;\;\; y \in \r
\eeqq
and Proposition \ref{prop_set_L} (iv) to check that for all $x\notin {\mathcal L} \cup \q$
 \beqq
 \lim\limits_{N\to +\infty} \frac{\ln | \{N x\} - \frac12|}{N} = 0.
 \eeqq
 Since for all $x \notin {\mathbb Z}+\frac12$ it is true that
\beqq
\big |\sec(\pi x) \big |< \frac{2}{|\{x\}-\frac12|},
\eeqq
we conclude that
 \beqq
 \lim\limits_{N\to +\infty} \frac{\ln | \sec(\pi N x)|}{N} = 0.  
 \eeqq
 Applying Theorem 1 from \cite{Baxa2002} 
we see that the above statement implies
\beqq
 \lim\limits_{k\to +\infty} \frac{1}{k} \sum\limits_{l=1}^k \ln  |\sec(\{\pi l x \} ) | = \ln(2),
 \eeqq 
which is equivalent to (\ref{asymptotics_sec}). 
Formula (\ref{asymptotics_csc}) can be verified in the same way; it also follows from Theorem 2 in \cite{HarLit1946}.
\end{proof}

\label{discussion1}
Using the results of Lemma \ref{lemma_trig_asymptotics} it is not hard to prove that both series in the right-hand side of (\ref{eqn_p_0_infty}) converge for
all $x>0$. However this does not guarantee that the series converges to $p(x)$. As an example, consider the asymptotic expansion
\beqq
f(x)=e^{-x} + \frac{1}{x-1} \sim \sum\limits_{n\ge 1} x^{-n}, \;\;\; x\to +\infty.
\eeqq
The asymptotic series in the right-hand side of the above equation converges for all $x>1$, but the limit is not equal to $f(x)$. Therefore, in order to prove 
Theorem \ref{thm_main} we would need to do some more work. 

Our main tool will be the Mellin transform of $S_1$, which is defined as
\beqq
\mm(w)=\e[S_1^{w-1}]=\int\limits_{\r^+} p(x) x^{w-1} \d x, \;\;\; \re(w)=1.
\eeqq
This function was studied in \cite{Kuz2010}, where it was proved that it can be analytically continued to a meromorphic function. There exists an explicit expression
(Theorem 8 in \cite{Kuz2010})
 for $\mm(s)$ in terms of the double gamma function, see \cite{Barnes1899} and \cite{Barnes1901} for the definition and properties of the double gamma function. It is also known that $\mm(s)$ satisfies several functional equations (see Theorem 7 in \cite{Kuz2010}) and if $\alpha \notin \q$ and $X \notin \cc_{k,l}$ it has simple poles at the points 
\beqq
 \{s^+_{m,n}\}_{m\ge 1, n\ge 1}=\{m+\alpha n\}_{m\ge 1, n\ge 1}, \qquad \{s^-_{m,n}\}_{m\ge 0, n\ge 0}=\{1-\alpha \rho -m- \alpha n\}_{m\ge 0, n\ge 0},
 \eeqq
with  residues given by 
\beqq
\textnormal{Res}\left(\mm(s) : s^+_{m,n} \right)=-b_{m-1,n}, \qquad \textnormal{Res}\left(\mm(s) : s^-_{m,n} \right)=a_{m,n},
\eeqq
see Lemma 2 in \cite{Kuz2010}.

\vspace{0.5cm}

{\it Proof of Theorem \ref{thm_main}:} In what follows we will always use the principal branch of the logarithm, which is defined in the domain $|\arg(z)|<\pi$ 
by requiring that $\ln(1)=0$. Similarly, the power function will be defined as $z^a=\exp(a\ln(z))$ in the domain $|\arg(z)|<\pi$.

Assume that $\alpha \notin {\mathcal L}\cup \q$, $\alpha \in (1,2)$ and define $c_k=1-\alpha \rho+\frac{\alpha}2-k$.
 Note that $c_k \ne s^{-}_{m,n}$ for all $k,m,n\ge 0$; if this was not true then $\alpha$ would be a rational number. Therefore $\mm(s)$ does not
have singularities on the lines $c_k+\i \r$ for all $k\ge 0$. 

Following the proof of Theorem 9 in \cite{Kuz2010}, we start with the inverse Mellin transform representation
\beqq
 p(x) = \frac{1}{2\pi \i} \int\limits_{1+\i \r} \mm(s) x^{-s} \d s, \;\;\; x>0.
\eeqq
The integral in the right-hand side of the above formula converges absolutely, since
$\mm(s)\to 0$ exponentially fast as $\im(s) \to \infty$, see Lemma 3 in \cite{Kuz2010}. 
Shifting the contour of integration $1+\i \r \mapsto c_k + \i \r$ and taking into account the residues at the poles $s=s^{-}_{m,n}$, we find that 
\beq\label{p_shift_contour}
p(x)=\sum \textnormal{Res}(\mm(s): s^-_{m,n}) \times x^{-s^-_{m,n}}+ \frac{1}{2\pi \i} \int\limits_{c_k+\i \r} \mm(s) x^{-s} \d s.
\eeq
where the summation is over all $m\ge 0, n \ge 0$, such that $s^-_{m,n}>c_k$. 
Now our goal is to prove that as $k\to +\infty$, the integral in the right-hand side of (\ref{p_shift_contour}) -- which we will denote by $I_k(x)$ -- converges to zero for every $x>0$.

First, we perform a change of variables 
\beq\label{eqn_Ix}
I_k(x)=\frac{1}{i}\int\limits_{c_k+\i \r} \mm(s) x^{-s} \d s= x^{-1+\alpha \rho-\frac{\alpha}2+k}\int\limits_{\r} \mm\left(1-\alpha \rho+\frac{\alpha}2-k+\i u\right) x^{-\i u} \d u.
\eeq 
Using equation (6.2) in \cite{Kuz2010} and the reflection formula for the Gamma function (see formula (8.334.3) in  \cite{Jeffrey2007}) we find that
\beqq
 \frac{\mm(s)}{\Gamma(s)\Gamma\left(\frac{1-s}{\alpha} \right)} = -\left[ \frac{\mm(s+1)}{\Gamma(s+1)\Gamma\left(\frac{1-(s+1)}{\alpha} \right)} \right]
 \frac{\sin\left(\frac{\pi}{\alpha} (1-s) \right)}
{\sin\left(\frac{\pi}{\alpha} (\alpha \rho-1+s) \right)}, \;\;\; s\in \c.
\eeqq
Iterating the above identity $k$ times gives us
\beq\label{mms_iteration}
\frac{\mm(s)}{\Gamma(s)\Gamma\left( \frac{1-s}{\alpha} \right)}=(-1)^k \left[ \frac{\mm(s+k)}{\Gamma(s+k)\Gamma\left( \frac{1-k-s}{\alpha} \right)} \right]
 \prod\limits_{j=1}^k \frac{\sin\left(\frac{\pi}{\alpha} (2-j-s) \right)}
{\sin\left(\frac{\pi}{\alpha} (\alpha \rho-2+j+s) \right)}, \;\;\; s\in \c. 
\eeq
Next, we use the fact that $\alpha \rho  \le 1$ for $(\alpha,\rho)\in \aa$, 
therefore $-1+\alpha \rho-\frac{\alpha}2<0$ while for $k\ge 2$ we have $-1+\alpha \rho-\frac{\alpha}2+k>0$, which in turn implies the estimate  
\beq\label{estimate_x}
x^{-1+\alpha \rho-\frac{\alpha}2+k}< (1+x)^k, \;\;\; k \ge 2, \;\;\; x > 0. 
\eeq 
 We express $\mm(s)$ in terms of $\mm(s+k)$ with the help of (\ref{mms_iteration}), substitute the resulting expression into (\ref{eqn_Ix}) and
use (\ref{estimate_x}) to
obtain the following estimate
\beq\label{eqn_estimate_1}
|I_k(x)| < (1+x)^{k} \int\limits_{\r} \big |\mm\left(1-\alpha \rho+\frac{\alpha}2+\i u\right) \big | \times |F_1(u;k)| \times |F_2(u;k)| \d u, \;\; k \ge 2, \;\;\; x>0, 
\eeq
where we have denoted
\beq\label{def_F1}
F_1(u;k)=\prod\limits_{j=1}^k \frac{\sin\left(\frac{\pi}{\alpha} \left( \alpha \rho - \frac{\alpha}2 - \i u+ k +1 -j \right) \right)}
{\sin\left(\frac{\pi}{\alpha} \left( \frac{\alpha}2 + \i u - k -1 + j\right) \right)}, 
\eeq
and
\beq\label{def_F2}
F_2(u;k)= \frac{\Gamma(1-\alpha \rho+\frac{\alpha}2+\i u-k)\Gamma\left( \rho-\frac{1}2-\i \frac{u}{\alpha}+\frac{k}{\alpha}\right)}
{\Gamma(1-\alpha \rho+\frac{\alpha}2+\i u)\Gamma\left(\rho-\frac{1}2-\i \frac{u}{\alpha} \right)}. 
\eeq

First, let us estimate $F_1(u;k)$ as $k\to +\infty$. Our goal is to prove that as $k\to +\infty$ the function $F_1(u;k)$ does not grow faster than some exponential function of $k$, and that this bound is uniform in $u\in \r$.  From the trigonometric identities
\beqq
  \sin(x+\i y)&=&\sin(x)\cosh(y) + \i \cos(x) \sinh(y), \\
  | \sin(x+\i y)|^2 &=& \cosh(y)^2 - \cos(x)^2
\eeqq
we see that $|\sin(x) | \cosh(y) \le |\sin(x+\i y)| \le \cosh(y)$, therefore
\beqq
\bigg |\frac{\sin(a+\i y)}{\sin(b+\i y)} \bigg | \le \frac{1}{|\sin(b)|}.
\eeqq
Combining the above estimate with (\ref{def_F1}) we see that
\beq\label{F_1_estimate_1}
|F_1(u;k)| \le   \prod\limits_{j=1}^k \bigg |\csc\left(\frac{\pi}{\alpha} \left(\frac{\alpha}2 - k -1 + j \right) \right) \bigg | 
=\prod\limits_{l=1}^k 
 \bigg |\sec\left(\frac{\pi l}{\alpha}  \right)\bigg |,
\eeq
where in the last step we have changed the index of summation $l=k+1-j$.  As we have established in the Proposition \ref{prop_set_L}, 
$\alpha \notin {\mathcal L} \cup \q$ implies $\alpha^{-1} \notin {\mathcal L} \cup \q$. Therefore the inequality (\ref{F_1_estimate_1}) 
and Lemma \ref{lemma_trig_asymptotics} tell us 
that there exists a constant $C_1=C_1(\alpha)>0$ such that 
\beq\label{F_1_estimate}
|F_1(u,k)|< C_1 3^k, \;\;\; k\ge 2, \;\;\; u \in \r. 
\eeq

Now we will deal with  $F_2(u;k)$ defined by (\ref{def_F2}). Our goal is to prove that as $k \to +\infty$ this function converges to zero faster than any 
exponential function of $k$, and that this happens uniformly in $u \in \r$. More precisely, we will prove that there exists a constant $C_2=C_2(\alpha)>0$ such that 
\beq\label{F2_estimate_final}
|F_2(u;k)| < C_2 (1+|u|) \exp\left(-\frac{\alpha-1}{\alpha} (k-1) \ln (k-1)+2k|A|\right), \;\;\; k \ge 2, \;\;\; u \in \r,
\eeq
where $A=(1+\ln(\alpha)-\alpha)/\alpha$.

Using the reflection formula for the gamma function we rewrite $F_2(u;k)$ as
\beqq
F_2(u;k)= (-1)^k g(u) \times \frac{\Gamma\left(  \rho-\frac{1}2-\i \frac{u}{\alpha}+\frac{k}{\alpha}\right)}
{ \Gamma \left( \alpha \rho - \frac{\alpha}{2} - \i u + k\right)},
\eeqq
where 
\beqq 
g(u)=\pi  \left[ \sin\left(\pi\left( \alpha \rho - \frac{\alpha}{2} - \i u\right) \right)
\Gamma\left(1-\alpha \rho + \frac{\alpha}{2} + \i u\right) \Gamma\left(\rho-\frac{1}2-\i \frac{u}{\alpha} \right) \right]^{-1}.
\eeqq
Next we use the following asymptotic expression (see formula (8.328.1) in \cite{Jeffrey2007})  
\beqq
|\Gamma(x+\i y)| \sim  \sqrt{2\pi} |y|^{x-\frac12} e^{-\frac{\pi}2 |y|}, \;\;\; y\to \infty, 
\eeqq
and the fact that $|\sin(x+\i y)|\sim \exp(|y|)/2$ as $y\to \infty$ to obtain 
\beq\label{big_asympt_estimate}
|g(u)|
 &\sim& \pi \left[ \frac12 e^{\pi |u|} \times \sqrt{2\pi} |u|^{\frac12-\alpha \rho+\frac{\alpha}2} e^{-\frac{\pi}2 |u|} \times
 \sqrt{2\pi} \bigg |\frac{u}{\alpha}\bigg |^{\rho-1} e^{-\frac{\pi}2 \frac{|u|}{\alpha}} \right]^{-1}  \\ \nonumber
&=&  |\alpha|^{\rho-1} |u|^{(\alpha-1)\left(\rho - \frac12\right)} e^{-(\alpha-1) \frac{\pi |u|}{2\alpha}},
\eeq
as $u\to \infty$.
We use the inequality $(\alpha-1)\left(\rho - \frac12\right)<\alpha-1<1$ and the fact that 
$g(u)$ is continuous in $u$ to  conclude
that there exists a constant $C_3=C_3(\alpha)>0$ such that 
\beq\label{big_asympt_estimate2}
g(u)
 < C_3 (1+|u|) e^{-(\alpha-1) \frac{\pi |u|}{2\alpha}}, \;\;\; u \in \r.
\eeq
Next, from the Stirling's asymptotic formula for the Gamma function (see formula (8.237) in \cite{Jeffrey2007}) we find that
\beqq
 \frac{\Gamma\left(\frac{s}{\alpha}\right)}{\Gamma(s)} = 
\sqrt{\alpha} \exp\left[ - s \left(  \frac{\alpha-1}{\alpha} \ln(s)  + A \right) + O(s^{-1})\right], \;\;\; s\to \infty, \; \re(s) \ge  0,
\eeqq
 where $A=(1+\ln(\alpha)-\alpha)/\alpha$ (note that $A<0$).
Again, the function in the left-hand side of the above equation is continuous in the half-plane $\re(s) \ge 0$; this fact and the above asymptotic formula
imply that  there exists a constant $C_4=C_4(\alpha)>0$ such that 
\beqq
\bigg | \frac{\Gamma\left(\frac{s}{\alpha}\right)}{\Gamma(s)} \bigg |  < C_4 
\bigg | \exp\left[ - s \left(  \frac{\alpha-1}{\alpha} \ln(s)  + A\right) \right] \bigg |, \;\;\; \re(s) \ge 0.
\eeqq
We set $s=\alpha \rho - \frac{\alpha}{2} - \i u + k$ (note that for $k \ge 2$ we have $\re(s)>0$)  and use (\ref{big_asympt_estimate2}) and the above estimate to conclude that for $k\ge 2$
\beq\label{eq1}
|F_2(u;k)| &=& |g(u)| \times \bigg | \frac{\Gamma\left(\frac{s}{\alpha}\right)}{\Gamma(s)} \bigg |  \\ \nonumber
 &<& C_5 (1+|u|) 
  \exp\left[-(\alpha-1)  \frac{\pi |u|}{2\alpha}+ \re \left\{ - s \left( \frac{\alpha-1}{\alpha} \ln(s)  + A \right) \right\}\right] \\ \nonumber
&=&
 C_5 (1+|u|) 
\exp \left[ -(\alpha-1)  \frac{\pi |u|}{2\alpha} - \re(s) \left(\frac{\alpha-1}{\alpha} \ln|s| + A \right) - u \frac{\alpha-1}{\alpha} \arg(s) \right] \\ \nonumber
&<&
C_5 (1+|u|)
\exp \left[ - \re(s) \left(\frac{\alpha-1}{\alpha} \ln|s| + A \right)\right]
\eeq
where $C_5=C_3 \times C_4$ and in the last step we have used the fact that $|\arg(s)|<\pi /2$. Using the following facts: (i)
$|s|>\re(s)>k-1$ and (ii) $\re(s)< 2k$  
we check that
\beq\label{eq2}
\exp \left[ - \re(s) \left(\frac{\alpha-1}{\alpha}\ln|s| + A \right)\right] < 
\exp \left[ - \frac{\alpha-1}{\alpha}  (k-1) \ln(k-1)  + 2k|A| \right]
\eeq
and combining (\ref{eq1}) and (\ref{eq2}) we finally conclude that  (\ref{F2_estimate_final}) is true with $C_2=C_5$.

Combining (\ref{eqn_estimate_1}), (\ref{F_1_estimate}) and (\ref{F2_estimate_final}) we see that 
\beqq
|I_k(x)|&<& C_6 \left( 3 (1+x) e^{2 |A|} \right)^k  e^{- \frac{\alpha-1}{\alpha}  (k-1) \ln(k-1)}, \;\;\; k \ge 2,
\eeqq
where
\beqq
 C_6=C_5 \times C_1 \times \int\limits_{\r}  \big | \mm\left(1-\alpha \rho+\frac{\alpha}2+\i u\right) \big | \times (1+|u|)  \d u. 
\eeqq
Therefore $I_k(x) \to 0$ as $k\to +\infty$ and we have proved the second series representation in
(\ref{eqn_p_0_infty}). The proof of the first series representation in  (\ref{eqn_p_0_infty}) is identical, except that we have to shift the contour of integration in the opposite direction. 

\qed

\vspace{0.5cm}
 Theorem \ref{thm_main} gives us a convergent series representation for $p(x)$ provided that $\alpha \notin {\mathcal L} \cup \q$. Let us consider what happens
in the remaining cases. 
When $\alpha \in \q$ and $X \notin \cc_{k,l}$ for all $k,l \in {\mathbb Z}$ then Theorems \ref{thm1} and  \ref{thm_main} can not be valid for the following reason.  Using Theorem 8 in \cite{Kuz2010} and properties of the double gamma function 
(see \cite{Barnes1899} and \cite{Barnes1901}) we conclude that in this case the function $\mm(s)$ will have multiple poles on the real line. This means that the asymptotic expansions 
(\ref{eqn_p_0_asympt}), (\ref{eqn_p_infty_asympt}) would include terms of the form
\beqq
{\textnormal {Res}}(\mm(s): s=w)\times \ln(x)^{j} x^{-w},\;\;\; 0 \le j \le k-1,
 \eeqq
 where $w$ is the pole of $\mm(s)$ of multiplicity $k$. This also means that the coefficients of such an expansion, given by the residues of $\mm(s)$,  will depend on the derivatives of the double gamma function, and in this case it will be 
 much harder (maybe even impossible) to evaluate them analytically.

 In the remaining case, when $\alpha \in {\mathcal L}$, it is not clear what happens with the series representation for $p(x)$ given in (\ref{eqn_p_0_infty}). By inspecting the form of coefficients $a_{m,n}$ and $b_{m,n}$ defined in 
(\ref{def_a_mn}) and (\ref{def_b_mn}) we see that the series in (\ref{eqn_p_0_infty}) might still converge, as the growth of the product of $\sin(\cdot)$ functions in the numerator in (\ref{def_a_mn})  might ``cancel'' the growth of the denominator. To prove the convergence one would need some results on the asymptotic behavior as $k \to +\infty$ of the following product
\beqq
 \prod\limits_{l=1}^k   |\csc(\pi (l x + y)  ) |.
\eeqq
Results in \cite{Baxa2002} give us information about these products in the case when $y \in \q$ and $x \notin {\mathcal L}$, 
but we were not able to find any general results in the existing literature. 
In any case, even if one were able to prove convergence of the series (\ref{eqn_p_0_infty}) for a more general class of parameter $\alpha$, it would not imply that the series converges to $p(x)$, see the discussion on page \pageref{discussion1}. To summarize, when $\alpha \in {\mathcal L}$ the question of
validity of Theorem \ref{thm_main} is open, and there are three possible scenarios: (i) the series in (\ref{eqn_p_0_infty}) do not converge, (ii) they converge, but not to $p(x)$, (iii) they converge to $p(x)$.




\end{document}